\documentclass[letterpaper, 10 pt, conference]{ieeeconf}  
\IEEEoverridecommandlockouts                              

  \let\labelindent\relax
\usepackage{amsmath} 
\usepackage{amssymb}  

\usepackage{tabulary}
\usepackage[english]{babel}
\usepackage{blindtext}
\usepackage[font=footnotesize]{caption}
\usepackage[makeroom]{cancel}
\usepackage{amsfonts}
\usepackage{pgfplots}
\usepackage{amsthm}  
\usepackage{amssymb}
\usepackage{graphicx}
\usepackage{subcaption} 
\usepackage{pgf,tikz}
\usepackage{commath}
\pgfplotsset{ 
  compat=newest, 
   legend style =
  {font=\footnotesize },
  label style = {font=\footnotesize},
every tick label/.append style={font=\footnotesize}
  }
\usepackage{tabularx}
\usepackage{float}
\usepackage{cite}
\usepgfplotslibrary{fillbetween}
  \usepgfplotslibrary{ternary}
\usetikzlibrary{shapes.multipart}
\usepackage{enumitem} 
\usepackage{bbm,xurl} 

\usepackage[strict]{changepage}

\newcommand{\vect}[1]{\boldsymbol{#1}}
\newcommand{\mat}[1]{\boldsymbol{#1}}

\allowdisplaybreaks

\DeclareMathOperator{\diag}{\text{diag}}

\DeclareMathOperator*{\argmax}{\rm{argmax}} 

\newtheorem{remark}{Remark}
\newtheorem{theorem}{Theorem}

\newtheorem{assumption}{Assumption}

\newtheorem{definition}{Definition}

\newtheorem{proposition}{Proposition}

\renewcommand{\eqref}[1]{Eq.~(\ref{#1})}  

\title{\LARGE \bf
Modelling the coevolution of opinion dynamics and decision making in social dilemmas}

\author{Ella C. Davidson, Lorenzo Zino, Ming Cao, and Mengbin Ye
\thanks{E. C. Davidson and M. Ye are with the Adelaide Data Science Centre, Adelaide University, Adelaide, Australia (\texttt{\{connor.davidson;ben.ye\}@adelaide.edu.au}). L. Zino is with the Department of Electronics and Telecommunications, Politecnico di Torino, Torino, Italy (\texttt{lorenzo.zino@polito.it}). M. Cao is with the Engineering and Technology Institute Groningen, University of Groningen, Groningen, The Netherlands (\texttt{m.cao@rug.nl}). E. C. Davidson is supported by the Australian Government with a Research Training Program scholarship. M. Ye is supported by the Australian Government through the Australian Research Council (DE250100199) and Office of National Intelligence (NI240100203).}
}

\begin{document}

\maketitle
\thispagestyle{empty}

\begin{abstract}
This paper proposes a mathematical model for the coevolution of actions and opinions for a population facing a social dilemma. In particular, we assume each person participates in a Public Goods Game (PGG), with their action being to cooperate or defect, and holds an opinion about which action they prefer. We propose a payoff function that combines the PGG with the Friedkin--Johnsen model from opinion dynamics to form a coevolutionary game. According to a discrete-time process, players asynchronously update their actions and opinions, aiming to maximise their individual payoff for the coevolutionary game using myopic best-response. We study the equilibria and provide conditions for the existence of the all-defection and all-cooperation consensus equilibria. We also establish conditions for global convergence to the all-defection equilibrium.
\end{abstract}

\section{Introduction}\label{sec:intro}
The emergence and persistence of widespread cooperation within a population when faced with a social dilemma is a fundamental problem of interest in biological, social and behavioural sciences~\cite{kollock_dilemmas_1998, dawes1980social}. This is partly because classical insights from game theory suggest that defection (also known as free riding) is often the best strategy from the perspective of a selfish rational individual. And yet, cooperation, which provides greater benefits for the entire population, is observed to persist in many real world scenarios. 

Collective action, i.e. the mobilisation of a large group to participate in a movement in pursuit of institutional or societal change, is a classical social dilemma~\cite{cantoni_protests_2024}. Participating in collective action is costly at an individual level and requires enough individuals to have a chance of success, but when successful, collective action often brings benefits to the entire population. More interestingly, there are cases of progressive movements (e.g. same-sex marriage and Black Lives Matter), in which many people participate even though they are not part of the repressed minority that would most benefit from the movements' success. 

Various factors that allow cooperation to emerge and persist over time have been proposed and studied. These include, and are not limited to, peer punishment~\cite{abbink_punish_2017}, reciprocity~\cite{gouldner_reciprocity_1960, park_opinions_2021}, and social influence~\cite{nowak_cooperation_2006}. While these factors are social in nature, a key psychological factor in determining a person's action is their own belief~\cite{henry2003beyond}. This is true for many contexts, including especially social dilemmas and collective action, as observed empirically within the literature on the social psychology of protest movements~\cite{thomas_protest_2024}. Moreover, generally, a protester's beliefs, attitudes, and psychological drivers (all of which evolve as the movement gains or loses traction) are central to their engagement with a movement~\cite{louis2022failure,nardini_BLM_2021}.

 

Despite empirical evidence of an interdependence between an individual's action and opinion, there have been limited efforts to combine the two into a unified mathematical model. The continuous-opinion discrete-action (CODA) model framework provided important initial contributions~\cite{martins_coda_2008,ceragioli_coda_2018}, operating under the assumption that an individual's action is a direct quantisation of their opinion, limiting the possibility of capturing the fundamental tension between cooperation and defection in social dilemmas. More recently, a mathematical model was proposed using a unified game-theoretic formulation, in which individuals played coordination games on a network and exchanged opinions on preferred strategies~\cite{aghbolagh_coevolutionary_2023}. The resulting coevolution of actions and opinions was studied analytically, and emergent phenomena such as unpopular norms and polarisation were identified~\cite{aghbolagh_coevolutionary_2023}. 
The model was extended by~\cite{liang_coevolution_2025} to consider an anti-coordination game instead. 

In this paper, we consider the coevolutionary dynamics of actions and opinions of a group of people in the context of social dilemmas. 
In more detail, individuals participate in a Public Goods Game (PGG), a classical game commonly used for modelling social dilemmas due to its flexibility and broad applicability~\cite{Bramoull2007}, and decide whether to cooperate or defect. An individual's strategic decision to cooperate or defect coevolves with their opinion (or attitude/belief) towards cooperation. In \cite{park_opinions_2021}, a PGG was integrated within a continuous-time opinion dynamics model, yielding a framework similar to CODA models. Players revise their opinion based on a realistic opinion dynamics model that accounts for social interactions and the payoff of a PGG, and then make decisions on their actions (the amount to contribute to the public good) based on their opinion in a stochastic manner. This, however, illustrates a one-way relationship, not an interdependent coevolution of actions and opinions.


Here, we adopt the game-theoretic coevolutionary approach~\cite{aghbolagh_coevolutionary_2023} and propose a joint payoff function for the coevolutionary game that combines i) the payoff from the PGG, ii) opinion dynamics evolving through social influence (inspired by the Friedkin--Johnsen model~\cite{friedkin_social_1990}), and iii) a term that captures an individual's desire for consistency between their action and opinion. We consider a population of players interacting on a network to share opinions, revising their actions and opinions asynchronously at discrete timesteps using myopic best-response updating. Our work contrasts the coevolutionary models in \cite{aghbolagh_coevolutionary_2023,liang_coevolution_2025}, which consider coordination and anti-coordination games rather than social dilemmas, and offers a novel lens to study the emergence and persistence of cooperation in social groups.

The main contributions of the paper are as follows. We propose a novel coevolutionary game, and obtain the explicit expression for the best-response updating, showing how a player's best-responding action and opinion are shaped by others' opinions and the PGG parameters. We study the equilibria, providing sufficient conditions for the existence of the all-defection and all-cooperation equilibria, and identify a sufficient condition for the all-defection equilibrium to be the unique equilibrium with a global stability property. The rest of the paper is organised as follows. Section~\ref{sec:model} presents the coevolutionary game and associated dynamics. Section~\ref{sec:results} contains the main results and Section~\ref{sec:conclusion} concludes the paper.


\section{The Model}\label{sec:model}

\subsection{Preliminaries}
The set of real and non-negative integer numbers are denoted by $\mathbb{R}$ and $\mathbb{Z}_+$ respectively. Bold lower case and upper case fonts denote a vector $\vect{x}$ and matrix $\vect{A}$, respectively, with the $i^{th}$ component being $x_i$ and the $j^{th}$ entry in the $i^{th}$ row being $a_{ij}$, respectively. The column vector of all ones and all zeros is denoted by $\vect{1}$ and $\vect 0$, respectively, and the identity matrix is $\vect{I}$; the appropriate dimension is understood within the context of its use. 

A weighted network is defined as a tuple $\mathcal{G} = (\mathcal{V}, \mathcal{E}, \mat W)$, where $\mathcal{V}=\{1,\dots,n\}$ is the set of $n$ nodes; $\mathcal{E}\subseteq\mathcal{V}\times\mathcal{V}$ is the edge set, whereby $(i,j)\in\mathcal{E}$ iff there is a link going from node $i$ to $j$, and $\boldsymbol{W}\in\mathbb{R}_+^{n\times n}$ is the weight matrix, so that entry $w_{ji}>0\Longleftrightarrow (i,j) \in \mathcal{E}$. A network is undirected iff $\vect{W}=\vect{W}^\top$. We assume that $\mat W$ is \textit{row-stochastic}, i.e., the entries of each row add to $1$, that is $\vect{W1}=\vect 1$.

A game $\Gamma(\mathcal{V}, \mathcal{A}, \boldsymbol{\pi})$ is defined by a set of players $\mathcal{V}=\{1,\dots,n\}$; a set of strategies $\mathcal{A}$ that the players can choose (here assumed to be the same for all agents); and a payoff vector $\boldsymbol{\pi}=[\pi_i, \dots, \pi_n]^\top$, where $\pi_i:\mathcal{A}^n\rightarrow\mathbb{R}$ is the payoff function for player~$i\in \mathcal{V}$. We let $z_i \in \mathcal{A}$ be the strategy of player $i$, which can be a scalar or vector and discrete or continuous. The vector $\vect{z}=[z_1,\dots,z_n]\in\mathcal A^{n}$ gathers the strategies of all players. Given a player $i\in\mathcal V$, we denote by $\vect{z}_{-i}=[z_1,z_2,\dots,z_{i-1},z_{i+1},\dots,z_n]\in\mathcal A^{n-1}$ the vector of the strategies of all players, except for player~$i$. The payoff that player~$i$ receives for selecting a strategy $z_i\in\mathcal A$ is typically dependent on the action of the others $\vect{z}_{-i}$, so it is denoted as $\pi_i(z_i,\vect{z}_{-i})$. 
Given a game, we define the set of best-response strategies and a Nash equilibrium as follows.
\begin{definition}[Best response and Nash Equilibrium]\label{def:BR}
    Given a game $\Gamma(\mathcal{V}, \mathcal{A}, \boldsymbol{\pi})$, the best-response strategies for player~$i\in \mathcal{V}$, given the state of the system $\vect{z}\in\mathcal A^n$ are defined as 
    \begin{equation}\label{eq:best_response}
        \mathcal{B}_i(\pi_i(\cdot,\vect{z}_{-i})):= \argmax\nolimits_{{z_i}\in \mathcal{A}} \ \vect \pi (z_i,\vect{z}_{-i}).
    \end{equation}
    A Nash equilibrium $\vect{z}^*=[z_1^*,\dots,z_n^*]\in\mathcal A^{n}$ is a strategy profile such that $z_{i}^* \in \mathcal{B}_i( \pi_i(\cdot,\boldsymbol{z}^*_{-i})),\, \forall\,i\in\mathcal V$.
\end{definition}


\subsection{The PGG and Actions}


We consider the set of players $\mathcal V$ participating in a PGG. Each player can decide whether to contribute to the public good or not, referred to as ``cooperate" or ``defect", respectively. We denote player $i$'s decision with the term $x_i \in \{0,1\}$, with $x_i=0$ and $x_i=1$ representing player~$i$ deciding to defect and cooperate, respectively. All $n$ player decisions are held in the action vector $\vect{x}=[x_1, \dots, x_n]^\top \in \{0,1\}^n$. 

If player~$i$ decides to cooperate, they contribute a unit amount to the public good. Regardless of whether they have cooperated or not, each player receives a reward, with the amount depending on how many players have cooperated. 
Thus, a player's payoff depends on the player's own action and the actions of the other $n-1$ players. Here, we consider a simple implementation of the PGG, with all players contributing a unit amount to the same public good if they cooperate, and a constant $r$, with $1<r<n$, that represents the public good multiplier~\cite{Bramoull2007, govaert_rationality_2021}. Thus,
\begin{align}\label{eq:payoff_PGG}
    \pi_i^{a}(x_i,\vect{x}_{-i}) &= x_i \Bigg[  \dfrac{r\left(\sum^n_{j=1, j\neq i}x_j + 1\right)}{n} -1 \Bigg] \notag \\ 
    &+ (1-x_i)\dfrac{r}{n}\sum\nolimits^n_{j=1, j\neq i}x_j,
\end{align}
where the first term (which is non-zero if player~$i$ cooperates) accounts for the share of the total reward received by player~$i$ if they cooperate, less their unit contribution), and the second term (which is non-zero if player~$i$ defects) accounts for the share of the total reward received by player~$i$ if they defect. 


\subsection{Opinion Dynamics}
To model the evolution of the players' opinions we opt for the use of the empirically validated Friedkin--Johnsen (FJ) model~\cite{friedkin_social_1990}. Each player~$i$ of a population $\mathcal V=\{1,\dots,n\}$ is assigned a continuous random variable $y_i\in [0,1]$, which represents player~$i$'s opinion. All $n$ of the players' opinions are held in the opinion vector $\vect{y}=[y_1, \dots, y_n]^\top \in [0,1]^n$. Players share their opinion on a social network $\mathcal G=(\mathcal V, \mathcal E,\mat W)$, where the entry $w_{ij}$ of $\mat W$ represents the social influence of player~$j$ on player~$i$. As detailed in the literature~\cite{proskurnikov_tutorial_2017}, the FJ model can be formulated as a game~\cite{ghaderi_opinion_2014}, where player~$i$'s payoff for adopting an opinion ${y}_i$, given an opinion vector $\vect y$ is defined as
\begin{equation}\label{eq:opinion}
    \pi_i^{o}(y_i,\vect{y}_{-i}) = -\dfrac{1}{2}(1-\gamma_i)\sum_{j \in \mathcal{V}}w_{ij}(y_i - y_j)^2 - \dfrac{1}{2}\gamma_i(y_i - u_i)^2,
\end{equation}
where the parameter $\gamma_i\in[0,1]$ represents player~$i$'s level of attachment to their constant prejudice $u_i\in[0,1]$. The best-response to this payoff yields the FJ opinion dynamics model~\cite{friedkin_social_1990}. Note that the FJ model, as standard in opinion dynamics, allows $w_{ii} > 0$ (self-loops in $\mathcal G$), but in a game-theoretic context corresponds to the non-standard concept of playing a game against oneself. This can be addressed by a minor adjustment to \eqref{eq:opinion}, see \cite{aghbolagh_coevolutionary_2023}. 

\subsection{Coevolutionary Model}

\subsubsection*{Coevolutionary Game} We are now in a position to introduce the coevolutionary game that combines the public goods and opinion dynamics game in an interdependent manner. In this setting, each player plays a PGG, where they can decide to defect, $x_i=0$, or to cooperate, $x_i=1$. Simultaneously, each player holds an opinion, $y_i \in [0,1]$, representing their level of support (or preference) for the action of cooperating in the PGG. Specifically, a player~$i$ fully supporting (preferring) defection and cooperation correspond to $y_i=0$ and $y_i=1$, respectively. We make the important distinction between whether a player chooses to defect or cooperate ($x_i$) and whether a player supports it ($y_i$), noting that these two may not always be aligned. Player~$i$'s strategy is defined as $z_i = (x_i, y_i)$, and we refer to the vector $\vect{z}= (\vect x, \vect y)$ as the \textit{state of the system} and strategy profile interchangeably.

Given the state of the system $\boldsymbol{z}=(\boldsymbol{x},\boldsymbol{y})$, we propose that the coevolutionary game has the following payoff function for player~$i$ selecting strategy $z_i = (x_i, y_i)$:
\begin{align}\label{eq:full_payoff}
    \pi_i(z_i,\vect{z}_{-i}) &= \alpha_i\pi_i^{a}(x_i,\vect{x_{-i}}) + \beta_i\pi_i^{o}(y_i,\vect{y_{-i}}) \notag \\ 
    & \qquad - \dfrac{1}{2}\lambda_i(x_i-y_i)^2,
\end{align}
with $\pi_i^{a}(x_i,\vect{x}_{-i})$ and $\pi_i^{o}(y_i,\vect{y}_{-i})$ from \eqref{eq:payoff_PGG} and \eqref{eq:opinion}, respectively. 
The parameters $\alpha_i,\beta_i,\lambda_i \in [0,1]$ represent player~$i$'s relative weight of actions, opinions, and desire to ensure consistency between their action and opinion (self-consistency), respectively. Without any loss in generality, we assume that $\lambda_i=1-\alpha_i-\beta_i$, from which \eqref{eq:full_payoff} can be readily interpreted as a convex combination of the three terms, one term proportional to the payoff of the corresponding PGG, one term proportional to the payoff of the corresponding opinion dynamics, and a third term that accounts for self-consistency. A major conceptual foundation of our model is to explicitly distinguish between action and opinion, while introducing the self-consistency term to reduce the total payoff as the player's action and opinion become less aligned.  All player payoffs for the coevolutionary game are stored in the vector $\vect{\pi}$, i.e. $\vect{\pi} = [\pi_1, \cdots, \pi_n]^\top$. We can now formally define the coevolutionary game as follows.

\begin{definition}[Coevolutionary game] \label{d:coev_game}
    The coevolutionary game is $\Gamma(\mathcal{V}, \mathcal{A}, \boldsymbol{\pi})$, played by the set of players $\mathcal V$, on $\mathcal G = (\mathcal{V},\mathcal E, \mat W)$, with strategy set $\mathcal A = \{0,1\}\times[0,1]$ and payoff vector function $\vect{\pi}$, where entry $\pi_i$ is defined in~\eqref{eq:full_payoff}.
\end{definition}

In order to characterise the best-response strategies to the coevolutionary game, we introduce a \textit{discriminant} quantity:
\begin{align}\label{eq:discriminant}
    \delta_i(z_i,\vect{z}_{-i}) &=  \alpha_i\left( \dfrac{r}{n}-1 \right) \\ \notag
    &+ \frac{\beta_i\lambda_i}{\beta_i+\lambda_i}\bigg(\gamma_iu_i+(1-\gamma_i)\sum_{j\in\mathcal{V}}w_{ij}y_j-\frac{1}{2}\bigg).
\end{align} 

The following result provides the best-response explicitly, by utilising the discriminant $\delta_i$. The proof is in the Appendix.
\begin{proposition}\label{prop:BR}
    Consider the coevolutionary game from Definition~\ref{d:coev_game}. Then,  
    given a state $\vect z\in\mathcal A^n$, for a player $i\in\mathcal V$, $(\bar x_i,\bar y_i)\in \mathcal{B}_i(\vect \pi_i(\cdot,\boldsymbol{z}_{-i}))$ if and only if
    \begin{align}
        \bar{x}_i &= \hat s_i(\delta_i(\vect{y})) \label{eq:BR_action} \\
        \bar y_i &= \dfrac{\beta_i(1-\gamma_i)\sum_{j\in\mathcal{V}}w_{ij}y_j+\beta_i\gamma_iu_i+\hat s_i(\delta_i(\vect y))\lambda_i}{\beta_i+\lambda_i},\label{eq:BR_opinion}
    \end{align}
    with 
    \begin{equation}
        \hat s_i(\delta_i(\vect{y})) = \left\{\begin{array}{ll}0&\text{if }\delta_i(\vect{y})\leq 0\\
        1&\text{if }\delta_i(\vect{y})\geq0.\end{array}\right.
    \end{equation}    
\end{proposition}

\begin{remark}\label{rem:br}When $\delta_i = 0$, the set of best responses comprises two elements: one with $\bar x_i=0$ and one with $\bar x_i=1$, and the corresponding values of $\bar y_i$, which differ, as per \eqref{eq:BR_opinion}.
\end{remark}
Finally, while the $\delta_i$ is presented in \eqref{eq:discriminant} as a function of $\vect z = (\vect x, \vect y)$, it turns out it is in fact only a function of $\vect y$. This property will play a key role in our analysis of the updating dynamics in the sequel. We note that~\eqref{eq:discriminant} suggests a one-way relationship between the players' opinions and actions, as we observed in~\cite {park_opinions_2021}. However, the formulation of the payoff in~\eqref{eq:full_payoff} does not. We conjecture that this results from the linear payoff formulation of the PGG; if a non-linear PGG payoff were implemented, this may no longer be the case.

\subsubsection*{Player Dynamics}
Let us now define the update dynamics for the players, who are repeatedly playing the coevolutionary game. At each discrete timestep $t \in \mathbb{Z}_+$, we posit that a set $\mathcal{R}(t) \subseteq \mathcal{V}$ of players becomes active and revises their strategy, and $\mathcal R(t)$ satisfies the following property.

\begin{assumption}\label{assumption:revision_sequence}
    (Revision sequence). There exists a constant $T<\infty$ such that $\cup_{s=0}^{T-1}\mathcal{R}(t+s) = \mathcal{V},$ for any $t \geq 0$.
\end{assumption}

\begin{remark}
    Assumption~\ref{assumption:revision_sequence} covers many synchronous and asynchronous update rules. Synchronous updating corresponds to $\mathcal{R}(t) = \mathcal{V}$ for all $t$, while asynchronous updating corresponds to $|\mathcal{R}(t)| = 1$, for all $t$. The condition of $\cup_{s=0}^{T-1}\mathcal{R}(t+s) = \mathcal{V}$ mimics typical assumptions in opinion dynamics with time-varying networks~\cite{proskurnikov_tutorial_2017}, and ensures that each player activates at least once within a time window of $T$, and that this occurs infinitely often as $t\to\infty$. 
\end{remark}

We assume that active players in $\mathcal R(t)$ will always update their future action and opinion by adopting the best-response strategy for the coevolutionary game given $\vect z(t)$. Hence, at timestep $t$, each player~$i \in \mathcal{R}(t)$ updates their state as
\begin{equation}
    z_i(t+1) \in \mathcal{B}_i(\pi_i(\cdot,\boldsymbol{z}_{-i}(t)))
\end{equation}
where $\mathcal{B}_i$ is the set of best-responses strategies computed in Proposition~\ref{prop:BR}. Here, we make a mild assumption that if a player's best response comprises two elements (see Remark~\ref{rem:br}), they will opt to defect due to inherit selfishness. Therefore, based on Proposition~\ref{prop:BR}, an active player~$i$ with state $z_i(t)=(x_i(t),y_i(t))$ is updated as 
\begin{align}
    x_i(t+1) &= s_i(\vect{y}(t)) \label{eq:update_action} \\
    y_i(t+1) &= \dfrac{\beta_i(1\!-\!\gamma_i)\sum_{j\in\mathcal V} w_{ij}y_j(t)\!+\!\beta_i\gamma_iu_i\!+\!s_i(\vect y(t))\lambda_i}{\beta_i+\lambda_i} \label{eq:update_opinion}
\end{align}
where
\begin{equation} \label{eq:update_s}
    s_i(\vect{y}(t)) = \left\{\begin{array}{ll}0&\text{if }\delta_i(\vect{y}(t))\leq 0\\
    1&\text{if }\delta_i(\vect{y}(t))>0,\end{array}\right.
\end{equation}
with $\delta_i(\vect{y}(t))$ defined in \eqref{eq:discriminant}. Any player $j \notin \mathcal{R}(t)$ is inactive, and thus $z_j(t+1)=z_j(t)$. 

We now point out some important features of the player update dynamics, and relevant interpretations within the context of the modelling problem. 

First, notice from \eqref{eq:update_action} that $x_i(t+1)$ is determined solely by $\delta_i$, which is a function of $\vect y(t)$. In other words, the players decision to defect or cooperate can potentially (but not necessarily) be influenced by other players' opinions, but not actions. However, $\delta_i$ also contains the parameters $r,n,\alpha_i$ which are associated with the PGG, and as we illustrate in the sequel, there are parameter values for which the sign of $\delta_i$ is independent of $\vect y$. This is a key departure from the coevolutionary model in \cite{aghbolagh_coevolutionary_2023}, which considered coordination games on networks with Friedkin--Johnsen opinion dynamics, and the equivalent $\delta_i$ that determined a player's action was dependent on both $\vect x$ and $\vect y$.

Second, a player's new opinion is a convex combination of their neighbours' current opinions, their existing prejudice, and their new action. This is a natural extension of the Friedkin--Johnsen model, which considers a convex combination of the former two quantities. It also illustrates how a player's new opinion is impacted by both their neighbours' current opinions and their own decision, illustrating the coevolving nature of the dynamics. Because $\alpha_i+\beta_i+\lambda_i = 1$, the dynamics are clearly influenced by the relative weighting that a player places on the PGG ($\alpha_i$), the opinion dynamics ($\beta_i$) and the need to be self-consistent ($\lambda_i$).

In summary, the \textit{coevolutionary game} $\Gamma=(\mathcal{V}, \mathcal{A}, \vect{\pi})$ is defined on a network $\mathcal{G}$ as in Definition~\ref{d:coev_game}. The \textit{coevolutionary model} is the dynamical system arising from the game being played repeatedly over discrete time-steps by the players, in which active players at each time-step update their action and opinion using a best-response approach. 

\section{Main Results}\label{sec:results}
Here, we study the equilibria of the system, characterizing them as Nash equilibria of the coevolutionary game, discussing the existence of some equilibria of interest (namely, those in which the entire population defect or cooperate, respectively), and show that under certain parameter conditions, there is global convergence to a unique equilibrium.

\subsection{Classification of Equilibria}
To characterise the equilibria of the coevolutionary model, we introduce some useful terminology. A state $\vect z^* = (\vect x^*, \vect y^*)$ is an equilibrium if \eqref{eq:update_action} and \eqref{eq:update_opinion} hold with $x_i(t+1) = x_i(t) = x_i^*$ and $y_i(t+1)=y_i(t)=y_i^*$ for all $i\in\mathcal V$.

Our results begin by relating the equilibria of the coevolutionary model to Nash equilibria of the coevolutionary game. We omit the proof for brevity, noting it can be obtained directly from Definition~\ref{d:coev_game} and Eqs.~(\ref{eq:update_action})--(\ref{eq:update_s}).

\begin{proposition}
   Under Assumption~\ref{assumption:revision_sequence}, the set of equilibria of the dynamics for the coevolutionary model in Eqs.~(\ref{eq:update_action})--(\ref{eq:update_s}) coincides with the set of Nash Equilibria of the coevolutionary game in Definition \ref{d:coev_game}. 
\end{proposition}



Henceforth, our analysis is written from the perspective of characterising the equilibria for the coevolutionary model, but as indicated above, this provides equivalent characterisation of the Nash equilibria for the coevolutionary game. In this paper, we derive results under the following assumption.
\begin{assumption}\label{a:0}
    All players have zero prejudice, i.e. $\gamma_i = 0, \forall i \in \mathcal{V}$, and $\alpha_i, \beta_i, \lambda_i \in (0,1), \forall i \in \mathcal{V}$.
\end{assumption}
This assumption allows us to the focus on the role of the variables $\alpha_i, \beta_i,$ and $\lambda_i$, which are the parameters of the PGG, opinion dynamics, and self-consistency. Extending the results to allow for $\gamma_i > 0$ is an important future direction, to examine the impact of attachment to existing prejudice.


To better interpret the results, let us define three types of consensus that can arise. 

\begin{definition}[Classification of consensus states]
    A state $\vect z = (\vect x, \vect y)$ is said to be an {action consensus} if $x_i=x_j \ \forall i,j \in \mathcal{V}$. For any arbitrary $\vect y$, we call the action consensus state $\vect z = (\vect 0, \vect y)$ as the all-defection state, and $\vect z = (\vect 1, \vect y)$ as the all-cooperation state. Similarly, we define an {opinion consensus} as $y_i=y_j \ \forall i,j \in \mathcal{V}$. Then, for an arbitrary $\vect x$, the state $\vect z = (\vect x, c\vect 1)$ is an opinion consensus state for any $c\in [0,1]$. Finally, we call $\vect z = (\vect 0, \vect 0)$ the {all-defection consensus} state, and $\vect z = (\vect 1, \vect 1)$ the {all-cooperation consensus} state.
\end{definition}

We begin our equilibria analysis by showing that any equilibrium in which there is action consensus must in fact have an opinion consensus.

\begin{proposition}\label{prop:consensus}
    Consider the coevolutionary model under Assumption~\ref{a:0}, and let $\vect{z}^*=(x^*,y^*)$ be an equilibrium of the model. There holds $\vect x^* = \vect 0 \Rightarrow \vect y^* = \vect 0$ and $\vect x^* = \vect 1 \Rightarrow \vect y^* = \vect 1$. That is, an equilibrium $\vect z^*$ with an action consensus must in fact also have an opinion consensus.  
\end{proposition}
\begin{proof}
    We first show that $\vect{x}^*= \vect{0} \Rightarrow~\vect{y}^*=\vect{0}$. At an equilibrium $\vect z^* = (\vect 0, \vect y^*)$, $x_i^* = s(\delta_i(\vect y^*)) = 0$ for all~$i$. Thus, \eqref{eq:update_opinion} at $\vect y^*$ evaluates, for all $i\in\mathcal V$, to be
    \begin{equation}
        y_i^* = \dfrac{\beta_i}{\beta_i+\lambda_i}\sum\nolimits_{j\in\mathcal{V}}w_{ij}y_j^*.
    \end{equation}
    Let us define $\beta_i/\beta_i+\lambda_i$ as the $i^{th}$ diagonal entry of the diagonal matrix $\mat\phi$. 
    We write the set of $y_i^*$ equilibria equations as $(\mat I- \mat{\phi}\mat{W})\vect{y}^* = \vect 0$, with
    $\vect y^* = \vect 0$ obviously a solution. To prove it is in fact the unique solution, we prove the matrix $\mat{I}-\mat{\phi}\mat{W}$ is invertible, yielding $\vect{y}^* = \left(\mat{I} - \mat{\phi}\mat{W}\right)^{-1}\vect 0$. Assumption~\ref{a:0} implies that $\beta_i/(\beta_i+\lambda_i) < 1$, and thus $\mat{\phi W}$ is a nonnegative matrix with all row sums strictly less than 1. Standard nonnegative matrix theory yields that $\rho({\mat \phi W}) < 1$ (e.g. by application of \cite[Corollary 8.1.29]{horn2012matrixbook}). Thus, $\mat I - \mat{\phi W}$ cannot have a zero eigenvalue and is in fact invertible.
    
    Next, we show $\vect{x}^*= \vect{1} \Rightarrow~\vect{y}^*=\vect{1}$. The hypothesis that $\vect x^* = \vect 1$ implies that $s_i(\delta_i(\vect y^*))=1$ must hold for all $i \in \mathcal{V}$. 
    Hence, at an equilibrium $\vect{y}^*$, there must hold
    \begin{equation}
        y_i^* = \dfrac{\beta_i}{\beta_i+\lambda_i}\sum\nolimits_{j\in\mathcal{V}}w_{ij}y_j^* + \dfrac{\lambda_i}{\beta_i+\lambda_i}.
    \end{equation}
    The set of equilibrium equations can thus be expressed as $\vect{y}^* = \mat{\phi}\mat{W}\vect{y}^* + (\mat{I} - \mat{\phi})\vect 1$. 
    The fact that $\mat W\vect 1 = \vect 1$ implies that $\vect y^* = \vect 1$ is a solution. We proved above that $\mat I - \mat{\phi W}$ is invertible, allowing us to rewrite $\vect y^* = (\mat I - \mat{\phi W})^{-1}(\mat{I} - \mat{\phi})\vect 1$. Clearly, the solution $\vect y^*$ is unique, and must be $\vect y^* = \vect 1$.
\end{proof}

\begin{theorem}\label{theorem:equilibria}
    Under Assumption~\ref{a:0}, the all-defection consensus state $(\vect{0},\vect{0})$ is always an equilibrium of the coevolutionary model. Moreover, it is the unique equilibrium of the coevolutionary model if there holds
    \begin{align}\label{eq:all_D_unique}
        \frac{\beta_i\lambda_i}{\beta_i+\lambda_i} \leq 2\alpha_i\left(1-\frac{r}{n}\right),
    \end{align}
    for all $i\in\mathcal V$. If instead
       \begin{align}\label{eq:all_coop}
        \frac{\beta_i\lambda_i}{\beta_i+\lambda_i} > 2\alpha_i\left(1-\frac{r}{n}\right),
    \end{align}
    for all $i\in\mathcal V$, then the all-cooperation consensus state $(\vect{1},\vect{1})$ is an equilibrium.
\end{theorem}
\begin{proof}
    To prove the first statement, we need to show that at an equilibrium $\vect z^* = (\vect x^*,\vect y^*)$, it is always true that i) $\vect x^* = \vect 0 \Rightarrow \vect y^* = \vect 0$, and ii) $\vect y^* = \vect 0 \Rightarrow \vect x^* = \vect 0$. The first part is immediate from Proposition~\ref{prop:consensus}. For the second part, observe that when $\vect y^* = \vect 0$, \eqref{eq:discriminant} evaluates to be
    \begin{equation}
        \delta_i(\vect y^*) = \alpha_i\bigg(\frac{r}{n}-1\bigg) - \frac{\beta_i\lambda_i}{2(\beta_i+\lambda_i)} \leq 0
    \end{equation}  
    because $r<n$ (recall $r$ is the public good multiplier). In other words, $x_i^* = 0$ for all $i$, hence $\vect z^* = (\vect 0,\vect 0)$ is always an equilibrium of the coevolutionary model.

    To prove the second statement, notice that if
    \begin{equation}\label{eq:01}
        \dfrac{\alpha_i(1-\alpha_i)\Big(1-\dfrac{r}{n}\Big)}{\beta_i\lambda_i}+\frac{1}{2} \geq1,
    \end{equation}
    then $\delta_i(\vect y^*) \leq 0$ holds independent of the particular value of $\vect y^*$, making $\vect z^*=(\vect 0,\vect 0)$ the unique equilibrium. Rearranging \eqref{eq:01} yields \eqref{eq:all_D_unique}. 

    To prove the all-cooperation consensus state is an equilibrium we must show that at an equilibrium $\vect z^* = (\vect x^*,\vect y^*)$, i) $\vect x^* = \vect 1 \Rightarrow \vect y^* = \vect 1$, and ii) $\vect y^* = \vect 1 \Rightarrow \vect x^* = \vect 1$. The first part is obtained from Proposition~\ref{prop:consensus}. For the second part, observe that under the assumption $\vect y^* = \vect 1$, \eqref{eq:discriminant} evaluates to be
    \begin{equation}\label{eq:02}
        \delta_i(\vect y^*) = \alpha_i\bigg(\frac{r}{n}-1\bigg) + \frac{\beta_i\lambda_i}{2(\beta_i+\lambda_i)}.
    \end{equation} 
    For a player to cooperate, $\delta_i(\vect y^*)>0$. Setting the condition in \eqref{eq:02} to be greater than zero and rearranging yields \eqref{eq:all_coop}. Therefore, if \eqref{eq:all_coop} holds for all $i \in \mathcal{V}$ then the all-cooperation consensus state is an equilibrium. 
\end{proof}

Interestingly, the conditions in \eqref{eq:all_D_unique} and \eqref{eq:all_coop} depict a trade-off between the role of self-consistency and opinion dynamics (left-hand side) and of the PGG (right-hand side). Stable cooperation seems possible when self-consistency and opinion dynamics can compensate for the fact that defection is inherently more advantageous for the single individual in the PGG. Note that when~\eqref{eq:all_D_unique} does not hold, there may be other equilibria besides the all-defection and all-cooperation consensus states, with a mixture of defectors and cooperators.



%

\subsection{Convergence}

Theorem~\ref{theorem:equilibria} identifies conditions on the individual-level parameters that imply that $\vect z^*=(\vect 0,\vect 0)$ is the unique equilibrium. Now, we prove that the same conditions yield convergence to said equilibrium for all initial conditions.
\begin{theorem}\label{theorem:convergence}
    Consider the coevolutionary model under Assumptions~\ref{assumption:revision_sequence} and \ref{a:0}. Suppose that $\mat W$ is symmetric and irreducible, i.e., $\mathcal G$ is undirected and connected. Suppose further that \eqref{eq:all_D_unique} holds for all $i\in\mathcal V$.
    Then, for all $(\vect x(0),\vect y(0))$, there holds $\vect x(t)=\vect 0$ for all $t\geq T$ and $\lim_{t\to\infty}\vect y(t)=\vect{0}$, where $T$ is given in Assumption~\ref{assumption:revision_sequence}.
\end{theorem}
\begin{proof}
In this proof, we first establish convergence of $\vect x(t)$ to $\vect 0$ in finite time. Then, we show that once $\vect x(t)$ has converged, $\vect y(t)$ will converge asymptotically to $\vect 0$.

Observe that if \eqref{eq:all_D_unique} holds for all agents, then $\delta_i(\vect y(t))\leq0$, for all $i \in \mathcal{V}$ and $t\geq 0$, irrespectively of $\vect{y}(t)$. In other words, if player~$i$ is active at timestep $t$, their best-response will involve defection for the PGG (i.e., $x_i(t+1)=0$).
Under Assumption \ref{assumption:revision_sequence}, after $T$ timesteps, there holds $\vect{x}(t) = \vect 0, \forall t \geq T$. In the following, without loss of generality, we focus on analysis of $\vect y(t)$ for $t\geq T$. 

Given that the action vector is at the all-defection action consensus state, $\vect x(t) = \vect 0$, for all $t\geq T$, the payoff that an arbitrary but fixed player~$i$ receives for having action $0$ and opinion $s\in [0,1]$ for any time $t\geq T$ simplifies to
\begin{equation}\label{eq:payoff_all_defect}
    \pi_i((0,s),\vect{z}_{-i}) = -\dfrac{1}{2}\bigg[ \beta_i \sum_{j\in \mathcal{V}} w_{ij}(s-y_j)^2 + \lambda_i s^2 \bigg].
\end{equation}
Now, consider the following potential function:
\begin{align}\label{eq:Phi_sum}
    \Phi(\vect y) = -\dfrac{1}{2} \sum_{i\in \mathcal{V}} \left[\sum_{j\in \mathcal{V}} \dfrac{w_{ij}}{2}(y_i-y_j)^2 + \dfrac{\lambda_i}{\beta_i}y_i^2\right].
\end{align}
Our hypothesis that $\mat W$ symmetric, i.e. $w_{ij}=w_{ji}$ for all $i,j \in \mathcal{V}$, means that $\Phi$ can be expressed as
\begin{align}\label{eq:Phi_sum_individual}
    &\Phi(y_i,\vect y_{-i}) = -\frac{1}{2}\bigg[\sum_{j\in \mathcal{N}_i} w_{ij}(y_i-y_j)^2 \bigg]  -\frac{1}{2}\frac{\lambda_i}{\beta_i}y_i^2 \notag \\
    &-\frac{1}{4}\sum_{k \in \mathcal{V} \setminus \{i\}}\sum_{\ell \in \mathcal{V} \setminus \{i\}} w_{k\ell}
    (y_k - y_\ell)^2 + \frac{1}{2}\bigg[ \sum_{k\in \mathcal{V}\setminus\{i\}} \frac{\lambda_k}{\beta_k}y_k^2 \bigg].
\end{align}
Notice here that we have simply rewritten $\Phi$ to highlight a fixed but arbitrary player $i \in \mathcal{V}$.

By exploiting \eqref{eq:payoff_all_defect} and \eqref{eq:Phi_sum_individual}, it can be verified that given two opinions, $s$ and $s'$ for player $i$, there holds
\begin{equation}\label{eq:payoff_potential}
    \pi(s,\vect{y}_{-i}) - \pi(s',\vect{y}_{-i}) = \beta_i(\Phi(s, \vect{y}_{-i}) - \Phi(s', \vect{y}_{-i})).
\end{equation}
Verify that the quadratic form of \eqref{eq:Phi_sum} is $\Phi(\vect y) = -\frac{1}{2} [ \vect y^\top(\mat I - \mat W + \mat{\Lambda})\vect y]$,
where $\mat{\Lambda} = \diag(\lambda_1/\beta_1, \cdots,\lambda_n/\beta_n)$ is a positive diagonal matrix. Note that $\mat I - \mat W + \mat \Lambda$ is a symmetric nonsingular $M$-matrix~\cite[Theorem~4.31]{qu2009cooperative_book}, and thus positive definite, which means $\Phi$ has a unique maximum equal to $0$, precisely at the maximiser $\vect y = \vect 0$.


Let player~$i$ be the active player at time $t\geq T$. From Definition~\ref{def:BR}, it is clear that the best-response dynamics implies that $\pi_i((0,y_i(t+1),\vect z_{-i}(t)) \geq \pi_i((0,y_i(t),\vect z_{-i}(t))$, with a strict inequality if $y_i(t+1) \neq y_i(t)$. This implies, from \eqref{eq:payoff_potential}, that $\Phi(\vect{y}(t+1))-\Phi(\vect{y}(t))\geq 0$, with a strict inequality if $y_i(t+1) \neq y_i(t)$. In other words, every timestep in which the active agent revises their opinion to a different value results in an increase in the potential $\Phi$. Since $\Phi(\vect y(t))$ is a monotonically non-decreasing sequence, and is bounded from above by $0$, the monotone convergence theorem~\cite{bartle1964elements} leads to the conclusion that $\lim_{t\to\infty} \Phi(\vect y(t)) = \Phi^*(\vect y^*)$. Our final step is to prove that $\Phi^* = 0$, or that $\vect y^* = \vect 0$. To see this, observe that in the limit $t\to\infty$, $\Phi(t+T) = \Phi(t)$, which is equivalent to $\vect y(t+T) = \vect y(t)$. From Assumption~\ref{assumption:revision_sequence}, this means that every agent has activated and has not changed their opinion value. In other words, $\vect y(t)$ has reached an equilibrium $\vect y^*$. We concluded earlier that $\vect x(t) = \vect 0$ for all $t\geq T$, and combined with Proposition~\ref{prop:consensus}, it follows that $\vect y^* = \vect 0$, which completes the proof.
\end{proof}


\section{Conclusion}\label{sec:conclusion}

This paper introduced a mathematical model that examined the coevolution of opinions and actions in social dilemmas. We established explicit expressions for a player's best-response for the coevolutionary game and analysed the model's equilibria, focusing on the all-defection and all-cooperation consensus states. A sufficient condition for global stability of the all-defection consensus equilibrium was identified. These promising results suggest several avenues of future research. First, we aim to establish an explicit region of attraction for the all-cooperation consensus equilibrium. Second, extending from Theorem~\ref{theorem:equilibria}, we are interested in identifying conditions for the existence of equilibria with a mixture of cooperators and defectors. We also aim to show that the dynamics will always converge to an equilibrium, as extensive simulations support this conjecture.

\appendix

\textit{Proof of Proposition~\ref{prop:BR}:} Note that as $x_i\in\{0,1\}$, we can determine the maximisers of $\pi_i$ by evaluating the maximum of the function when $x_i=0$ and $x_i=1$, and obtain the global maximizer (i.e., the best-response) by comparing the two values obtained. We first consider the case where $x_i=1$, and differentiate $\pi_i$ with respect to $y_i$ to obtain
\begin{align}
    \pi'_i((1,y_i),\vect{z}_{-i}) = &-\beta_i(1-\gamma_i)\bigg(y_i\sum_{j\in\mathcal{V}}w_{ij} -\sum_{j\in\mathcal{V}}w_{ij}y_j\bigg) \notag\\ 
    &\qquad -\beta_i\gamma_i(y_i-u_i)-\lambda_i(y_i-1).
\end{align}
Since $\pi''_i((1,y_i),\vect z_{-i}) = -\beta_i-\lambda_i<0$, it follows that $\pi_i((1,y_i),\vect z_{-i})$ is strictly concave and has a unique maximum, allowing us to solve $\pi'_i(1,y_i) = 0$ in order to find the maximiser $y_i^1$. The same process can be followed for $x_i=0$, and we can solve $\pi'_i(0,y_i) = 0$ to find the maximiser $y_i^0$. Define $h=\beta_i(1-\gamma_i)\sum_{j\in\mathcal{V}}w_{ij}y_j+\beta_i\gamma_i u_i$ and the two maximisers can be written as
\begin{align}
    y_i^{1}=\dfrac{\beta_ih+\lambda_i}{\beta_i+\lambda_i}\;\;\text{and}\;\;
    y_i^{0}=\dfrac{\beta_ih}{\beta_i+\lambda_i}.
\end{align}
Evidently, the payoff to cooperate and defect is maximised by adopting the opinion $y_i^1$ and $y_i^0$ respectively. It follows that the best-response corresponds to the greater of $\pi_i((1,y_i^{1}),\vect{z}_{-i})$ versus $\pi_i((0,y_i^{0}),\vect{z}_{-i})$. Verify that the discriminant $\delta_i(z_i,\vect{z}_{-i})$ in \eqref{eq:discriminant} is precisely $\pi_i((1,y_i^{1}),\vect{z}_{-i})-\pi_i((0,y_i^{0}),\vect{z}_{-i})$.
Thus, $\delta_i(z_i,\vect{z}_{-i}) > 0$ and $\delta_i(z_i,\vect{z}_{-i}) < 0$ indicate that $(1,y_i^{1})$ and $(0,y_i^{0})$ are the best-response strategies, respectively. If $\delta_i(z_i,\vect{z}_{-i}) = 0$, both $(1,y_i^{1})$ and $(0,y_i^{0})$ are best-responses. \hfill $\qed$

\bibliographystyle{IEEEtran}
\bibliography{ref}

\end{document}